\documentclass[11pt]{article}
\usepackage{silence}
\usepackage[T1]{fontenc}
\usepackage[utf8]{inputenc}
\usepackage{lmodern}
\usepackage{amsmath,amssymb,amsthm}
\usepackage{microtype}
\usepackage[margin=1.15in]{geometry}
\usepackage[hidelinks,pdftitle={The gate of self-address: where decidable adjudication ends},pdfauthor={Platon Sifnaios}]{hyperref}

\newtheorem{theorem}{Theorem}[section]
\newtheorem{proposition}[theorem]{Proposition}
\newtheorem{corollary}[theorem]{Corollary}
\newtheorem{lemma}[theorem]{Lemma}
\theoremstyle{definition}
\newtheorem{definition}[theorem]{Definition}
\newtheorem{question}[theorem]{Question}
\theoremstyle{remark}
\newtheorem{remark}[theorem]{Remark}
\newtheorem{convention}[theorem]{Convention}

\newcommand{\M}{\mathbf{M}}
\newcommand{\NN}{\mathbb{N}}
\newcommand{\mk}{\langle\,\rangle}
\newcommand{\Led}{\mathbf{L}}
\newcommand{\ask}{\textup{\scshape ask}}
\newcommand{\PS}{\textup{\scshape PSpace}}

\title{The gate of self-address:\\ where decidable adjudication ends}
\author{Platon Sifnaios\\ \small Independent researcher, Athens, Greece\\ \small p.sifnaios@gmail.com}
\date{}

\begin{document}
\maketitle

\begin{abstract}
An \emph{annulment structure} consists of a numbered domain of
\emph{distinctions} together with a $\Sigma^0_1$ \emph{manifestation}
predicate; an \emph{adjudicator} is a partial map assigning to
distinctions the verdicts \emph{annulled} or \emph{exempt}. We ask
which domains admit an adjudicator that is total, correct, and complete
in its mission --- that is, annulling every non-manifesting member
of the domain --- and we locate the boundary exactly. On the positive
side, domains with decidable manifestation admit canonical adjudicators
with certificates, instantiated here for Presburger-conditioned
distinctions and for finite-state re-entry systems. On the negative
side, adjoining to such a domain a single nullary gate, by which a
distinction may query the verdict passed on itself, destroys
decidability uniformly: the resulting domain carries a trichotomy in
which every adjudicator fails totality, exhaustiveness, or soundness
at a single distinction fixed in advance.
Bounding the depth of self-address refines this: each finite level
remains decidable, the hierarchy of iterated verdicts is the
synchronous update of a Boolean network on the query graph, deciding
whether it stabilises is \PS-complete for explicitly
presented networks, and periods as
large as $2^n-1$ occur at closure size $n$. What fails in the limit is
therefore convergence rather than decidability, and the diagonal
distinctions are exactly the points of oscillation, with period two;
their mean frequency of $1/2$ is precisely the value a reflective
oracle is forced to return there.
The boundary admits a second, degree-theoretic formulation: the least
Turing degree of an adjudicator with all three properties, over the
standard domain relative to an oracle $X$, is exactly $\deg(X')$, so
adjudication costs precisely one jump at every level.
The two formulations describe one phenomenon, and together they show
that the boundary is a three-way trade among determinacy, correctness,
and residence at the level adjudicated; the reflective oracles of
Fallenstein, Taylor and Christiano occupy the remaining corner. The
underlying fixed-point core is not tied to the Kleene numbering: a
dichotomy holds over every precomplete numbering in the sense of
Ershov, and the third case is available exactly where divergence has an
unmarked home. We situate the framework relative to the
categorical treatment of diagonal arguments, where it falls on the
intensional side of the divide between Kleene's two recursion
theorems. G\"odel's incompleteness theorems are nowhere used.
\end{abstract}

{\small

\PS-completeness; reflective oracle.}

\section{Introduction}\label{sec:intro}

This paper asks a question with a sharp answer: given a numbered
family of \emph{distinctions} --- items with a semi-decidable success
condition, here called \emph{manifestation} --- when can a single
effective procedure correctly and totally \emph{annul} exactly the
distinctions that fail to manifest? The answer is that this is
possible for a wide and natural class of domains, and that it fails
uniformly the moment one syntactic feature is present. Locating that
feature is the point of the paper.

The positive side comes first (Section~\ref{sec:decidable}). Domains
with decidable manifestation admit \emph{canonical} adjudicators that
are total, correct, and mission-complete, and that certify
each verdict by a finite object. Two instances are given: distinctions
conditioned on sentences of Presburger arithmetic, and finite-state
re-entry systems in the calculus of indications, whose transient and
period furnish the certificate. Neither language contains any
construct by which a distinction may refer to a verdict.

Adjoining exactly one such construct --- a nullary gate
$\ask$, evaluated at the verdict passed on the very system in
which it occurs --- destroys decidability, uniformly and at every
oracle level (Theorem~\ref{thm:cross}). The extended domain carries a
\emph{trichotomy}: every candidate adjudicator is defeated at a
single distinction, fixed in advance and the same for all, and the
defeat is exhaustively one of unsoundness, non-exhaustiveness, or
partiality
--- the failure of (S), of (E), or of (T). The
same trichotomy holds over the standard domain
(Theorem~\ref{thm:tri}), of which the gate construction is the
syntactic localisation.

Bounding the gate refines the picture considerably
(Section~\ref{sec:bounded}). Let the gate carry a parameter, so that a
distinction may query the verdict on any distinction rather than only
on itself, and let $a_k$ be the verdict function obtained by allowing
depth $k$ of iterated query. Each $a_k$ is total computable, and the
passage from $a_k$ to $a_{k+1}$ is the synchronous update of a Boolean
network on the query graph. Every finite level is therefore decidable;
deciding whether the sequence stabilises is \PS-complete
for explicitly presented networks;
and periods as large as $2^n-1$ are realised at closure size $n$,
inside the re-entry language itself. The crossing is accordingly not a
discontinuity but a limit, and what fails at the limit is convergence
rather than decidability. Under pure self-address the period is always
one or two, and it is two exactly at the diagonal distinctions; their mean
frequency there, $1/2$, is precisely the value a reflective oracle is
forced to return at the liar machine, which suggests reading
its randomisation as the assignment of a limiting frequency
(Remark~\ref{rem:mean}).

The boundary has a further formulation, in degrees rather than syntax
(Section~\ref{sec:jump}): the least Turing degree of a correct, total,
mission-complete adjudicator of the $X$-standard domain is exactly
$\deg(X')$.
Adjudication costs precisely one jump, at every level. Read together
with the crossing, this exhibits the boundary not as a wall but as a
three-way trade among three desiderata --- that the verdict be
definite, that it be correct, and that the adjudicator reside at the
level it adjudicates. Any two are available; all three are not. The
reflective oracles of \cite{FTC15} occupy the corner obtained by
surrendering definiteness, and we say so explicitly
(Remark~\ref{rem:reflective}), since the impossibility proved here is
relative to constraints that other work has found it profitable to
relax.

The remaining sections support and delimit this picture. The effective
content of the framework is determined exactly: the manifestation
problem is $1$-complete, creative, and computably isomorphic to the
halting problem; its complement is $\Pi^0_1$-complete and productive;
and manifestation cannot even be computably separated from explicit
exemption, so that the failure is not merely of decision but of
approximation (Section~\ref{sec:structure}). The adjudicator
properties and the three failure modes are classified in the
arithmetical hierarchy (Section~\ref{sec:indexsets}). The fixed-point
core is shown not to be an artefact of the Kleene numbering: a
dichotomy holds over every precomplete numbering in the sense of
Ershov, and the third horn --- the pricing of silence --- is available
exactly where divergence has an unmarked home, namely over complete
numberings with an unmarked special element
(Remark~\ref{rem:numberings}). Finally, we situate the framework
against the categorical treatment of diagonal arguments
(Section~\ref{sec:lawvere}); it falls on the intensional side of the
divide between Kleene's two recursion theorems, and what it offers
that those treatments do not is a boundary rather than a further
unification of impossibility.

Methodologically the ingredients are classical --- Kleene's recursion
theorem, Post's completeness and productivity, the isomorphism
theorems of Myhill and Smullyan, Ershov's theory of numberings --- and
we make no claim of difficulty for the individual proofs. The
contribution is the boundary itself, stated as a closure condition on
presented domains; its level-invariance
(Corollary~\ref{cor:rel}); and its two equivalent formulations, in
syntax and in degrees. The framework formalises a question with a long
informal history outside computability theory; nothing below depends
on that history, and the paper is self-contained. We note once that
G\"odel's incompleteness theorems are used nowhere: the entire
development runs on the recursion theorem and the elementary theory of
c.e.\ sets.

\section{Preliminaries and the annulment framework}\label{sec:prelim}

Fix a standard acceptable numbering $(\varphi_e)_{e\in\NN}$ of the
unary partial computable functions, with the s-m-n and
universal-machine properties; $W_e=\operatorname{dom}\varphi_e$. Our
notation for computability-theoretic notions ($\leq_m$, $\leq_1$,
$\leq_T$, creative and productive sets, the jump $X'$, the classes
$\Sigma^0_n,\Pi^0_n$ and their relativisations) follows
\cite{Rog67,Soa87}. $K=\{e:\varphi_e(e)\!\downarrow\}$ and
$K_0=\{\langle e,x\rangle:\varphi_e(x)\!\downarrow\}$ denote the
halting sets.

\begin{convention}\label{conv:mark}
A distinguished output value $1$ is called the \emph{mark}, written
$\mk$. All programs are evaluated at input $0$.
\end{convention}

\begin{definition}[Annulment structure]\label{def:structure}
A \emph{distinction} is an index $p\in\NN$. The \emph{manifestation
set} is
\[
\M \;=\; \{\,p\in\NN : \varphi_p(0)\!\downarrow = 1\,\},
\]
and the \emph{ledger} is $\Led=\NN\setminus\M$. A \emph{domain} is a
pair $(D,\M\!\restriction\! D)$ with $D\subseteq\NN$; the
\emph{standard domain} is $(\NN,\M)$. An \emph{adjudicator} is an index $a$, acting on distinctions through
$\varphi_a$: $a$ \emph{annuls} $p$ if $\varphi_a(p)\!\downarrow=1$,
\emph{exempts} $p$ if $\varphi_a(p)\!\downarrow\neq1$, and is
\emph{silent} on $p$ if $\varphi_a(p)\!\uparrow$.
\end{definition}

\begin{definition}[The adjudicator properties]\label{def:TUV}
Relative to a domain $D$, an adjudicator $a$ is:
\begin{itemize}
\item[(T)] \emph{total} if it is silent on no $p\in D$;
\item[(E)] \emph{exhaustive} if it annuls every $p\in D$;
\item[(M)] \emph{mission-complete} if it annuls every
$p\in D\setminus\M$;
\item[(S)] \emph{sound} if $\varphi_a(p)\!\downarrow=1$ implies
$p\notin\M$, for $p\in D$.
\end{itemize}
A total adjudicator is \emph{correct} for $D$ if it annuls exactly
the non-manifesting members of $D$ --- equivalently, if it is both
sound and mission-complete. The \emph{canonical adjudicator} of $D$
is the (set-theoretic) function $\chi_D(p)=1$ iff
$p\in D\setminus\M$; it is total on $D$ and correct.
\end{definition}

The mark is deliberately minimal. The notation $\mk$ is Spencer-Brown's
cross and denotes the marked state; it is not the empty multiset, which
denotes the void. Readers of \cite{Sif26} may take the mark as the
form $\langle\varnothing\rangle$, the first non-empty hereditarily
finite multiset, under the identification of forms with such multisets
proved there (Theorem~8.2); nothing below depends on this reading.

\section{The trichotomy}\label{sec:tri}

\begin{theorem}[Trichotomy]\label{thm:tri}
There is a total computable function $d$ such that for every index
$a$, writing $\delta_a:=d(a)$:
\begin{enumerate}
\item\label{tri:fix} $\delta_a\in\M \iff
\varphi_a(\delta_a)\!\downarrow=1$;
\item\label{tri:cases} exactly one of the following holds:
\emph{(i)} $\varphi_a(\delta_a)\!\downarrow=1$;
\emph{(ii)} $\varphi_a(\delta_a)\!\downarrow\neq1$;
\emph{(iii)} $\varphi_a(\delta_a)\!\uparrow$.
\end{enumerate}
Consequently no adjudicator over the standard domain is simultaneously
total, exhaustive, and sound: \emph{(i)} refutes \emph{(S)},
\emph{(ii)} refutes \emph{(E)}, \emph{(iii)} refutes \emph{(T)}.
\end{theorem}

\begin{proof}
Let
\[
\theta(a,x)=
\begin{cases}
1 & \text{if }\varphi_a(x)\!\downarrow=1,\\
0 & \text{if }\varphi_a(x)\!\downarrow\neq1,\\
\uparrow & \text{if }\varphi_a(x)\!\uparrow.
\end{cases}
\]
By s-m-n fix total computable $s$ with
$\varphi_{s(a,x)}(y)=\theta(a,x)$ for all $y$; by the recursion theorem
with parameters \cite{Kle38,Rog67} fix total computable $d$ with
$\varphi_{d(a)}=\varphi_{s(a,d(a))}$ for all $a$. Then
$\varphi_{\delta_a}(0)=\theta(a,\delta_a)$, giving \eqref{tri:fix}; the
three cases of $\theta$ are mutually exclusive and jointly exhaustive,
giving \eqref{tri:cases}. If (T) and (E) hold then case (i) obtains,
refuting (S).
\end{proof}

\begin{remark}[What the theorem does and does not say]\label{rem:trivial}
The bare conjunction (T)$\wedge$(E)$\wedge$(S) over the standard domain
is refutable without fixed points: (E) implies (T) outright, and any
fixed index $p_0$ of the constant-mark program witnesses the failure
of (E)$\wedge$(S), since an exhaustive adjudicator annuls $p_0\in\M$.
The content of Theorem~\ref{thm:tri} is therefore not the
inconsistency of the conjunction but the conjunction of three stronger
facts: the classification of \emph{arbitrary} $a$ --- total or
partial, exhaustive or selective --- into exactly one of the three
horns; the uniformity and (Corollary~\ref{cor:rel})
oracle-independence of the witness map $d$; and the
\emph{verdict-conditioned} character of the witness ---
$\delta_a\in\M$ exactly if $a$ annuls $\delta_a$ --- which no
constant witness provides and on which everything that follows
depends: Corollaries~\ref{cor:nodecider} and~\ref{cor:audit} (neither
of which follows from the constant witness), the relativisation, the
crossing (Theorem~\ref{thm:cross}), and the results of
Remark~\ref{rem:numberings}. In particular
Corollary~\ref{cor:nodecider} concerns adjudicators that annul
\emph{selectively and correctly}; against these the constant witness
has no purchase.
\end{remark}

\begin{remark}\label{rem:rca}
The construction is a finite manipulation of indices; the proof
formalises without change in $\mathsf{RCA}_0$, and the map $d$ may be
taken primitive recursive.
\end{remark}

\begin{corollary}\label{cor:nodecider}
There is no total $a$ with $\varphi_a(p)\!\downarrow=1\iff p\notin\M$
for all $p$; equivalently, $\M$ is undecidable.
\end{corollary}

\begin{proof}
Instantiate \eqref{tri:fix}:
$\varphi_a(\delta_a)=1\iff\delta_a\in\M\iff\varphi_a(\delta_a)\neq1$.
\end{proof}

\begin{corollary}[Failure of self-audit]\label{cor:audit}
There is a total computable $c$ such that $\sigma_a:=c(a)$ satisfies
$\sigma_a\in\M\iff\exists p\,(\varphi_a(p)\!\downarrow=1\wedge
p\in\M)$. If $a$ is total and exhaustive, then $\sigma_a\in\M$
while $a$ annuls $\sigma_a$.
\end{corollary}

\begin{proof}
The matrix is $\Sigma^0_1$ uniformly in $a$; let $\sigma_a$ dovetail
the search, outputting $1$ on success and diverging otherwise. Under
(T)$\wedge$(E), case (i) of Theorem~\ref{thm:tri} supplies the witness
$\delta_a$.
\end{proof}

\begin{corollary}[Relativisation]\label{cor:rel}
Let $X\subseteq\NN$ and let $(\varphi^X_e)$ be the standard enumeration
of $X$-partial computable functions, $\M^X=\{p:\varphi^X_p(0)\!\downarrow=1\}$.
The \emph{same} function $d$ of Theorem~\ref{thm:tri} satisfies, for
every $a$ and every $X$:
$\delta_a\in\M^X\iff\varphi^X_a(\delta_a)\!\downarrow=1$, with the
trichotomy verbatim. Hence for every oracle $X$, no $X$-computable
adjudicator over the $X$-standard domain is simultaneously total,
exhaustive, and sound.
\end{corollary}

\begin{proof}
The s-m-n index function $s$ manipulates program text only and is
oracle-independent, and the parametrised recursion theorem derived from
it inherits this independence \cite[chs.~9 and~11]{Rog67}. The defining program
text of $\theta$ is fixed, its behaviour varying with $X$; hence
$\varphi^X_{d(a)}(0)=\theta^X(a,d(a))$ for all $X$ simultaneously, and
the proof of Theorem~\ref{thm:tri} goes through unchanged.
\end{proof}

\begin{remark}[The core is not tied to the Kleene numbering]\label{rem:numberings}
The fixed-point core of Theorem~\ref{thm:tri} is an instance of a more
general phenomenon, which we record without proof since nothing below
uses it. Let $\gamma:\NN\to S$ be a numbering and $m:S\to\{0,1\}$ any
map attaining both values, with no effectivity assumed. If $\gamma$ is
precomplete in the sense of Ershov \cite{Ers73}, then for every total
computable $A:\NN\to\{0,1\}$ there is $n$ with $m(\gamma(n))=A(n)$:
Ershov's fixed point theorem applied to the map sending $x$ to an
index of a witness for $A(x)$. This is the (i)/(ii) dichotomy. The
third case requires more: it requires that divergence be sent
somewhere unmarked, which is available when $\gamma$ is complete with
special index $u$ satisfying $m(\gamma(u))=0$. The Kleene numbering is
complete with the totally undefined function as special element, which
is why Theorem~\ref{thm:tri} has three cases rather than two. Whether
the three-case form conversely forces such an element is, as far as we
know, open, and we do not pursue it here.
\end{remark}

\section{Decidable domains, certificates, and the crossing}\label{sec:decidable}

\begin{definition}\label{def:decfield}
A domain $D$ is \emph{decidable} if the restricted manifestation
problem $\{p\in D:p\in\M_D\}$ is decidable, where $\M_D$ is the
manifestation predicate appropriate to the presentation of $D$. A
\emph{certificate scheme} for $D$ is a computable relation
$C(p,w)$ of distinctions and finite objects such that $p\notin\M_D\iff
\exists w\,C(p,w)$ with a computable witness function.
\end{definition}

\begin{proposition}\label{prop:canonical}
If $D$ is decidable, the canonical adjudicator $\chi_D$ is total
computable and correct. Two instances:
\begin{itemize}
\item[\emph{(a)}] \emph{Presburger-conditioned distinctions}: $D$
indexed by sentences $\phi$ of Presburger arithmetic, with
$d_\phi\in\M_D\iff\NN\models\phi$. Decidability by
\cite{Pre29}; a decision procedure with certificates is obtained by
the B\"uchi--Elgot automata method \cite{Buc62,Elg61,WB95}, the
accepting or rejecting run of the associated automaton serving as the
witness.
\item[\emph{(b)}] \emph{Finite-state re-entry systems} in the calculus
of indications \cite{SB69,Var75,KV80}: systems of $n$ feedback
equations over $\{\text{void},\text{mark}\}$ built from crossing and
juxtaposition, observed at one coordinate. The state space has size
$\leq 2^n$ and the evolution is deterministic, so the behaviour is
ultimately periodic; the pair $(\mu,\pi)$ of transient and
period is a certificate scheme: $p\notin\M_D$ iff the
prefix-plus-cycle of length $\mu+\pi$ is mark-free, and
$\mu+\pi\leq2^n$ bounds the certificate effectively in the
presentation.
\end{itemize}
\end{proposition}

\begin{proof}
Immediate from the definitions and the cited decidability results.
\end{proof}

\begin{remark}[On the choice of the re-entry formalism]\label{rem:formalism}
A finite-state re-entry system is precisely a synchronous Boolean
network --- equivalently, an input-free sequential circuit ---
presented in the primary-algebra signature: over the two-element
algebra $\{\text{void},\text{mark}\}$, crossing is negation and
juxtaposition is disjunction, and the translation is the identity on
semantics. The Spencer-Brown presentation is thus a choice of
vocabulary, not of mathematical content, and the reader who prefers
Boolean networks loses nothing by making the substitution throughout.
It is retained here only because it makes the absence of a
verdict-querying construct visible in the syntax: the primary algebra
has exactly two operations, and neither can name an adjudicator.
\end{remark}

The two instances share a feature that the next theorem shows to be
essential: their languages have no construct by which a distinction
may query the verdict of an adjudicator. Adjoining that single
construct destroys decidability, uniformly. Making this precise
requires some care, since the gate is to be evaluated at the index of
the system in which it occurs, and that index is fixed only once the
extended language has been numbered. We therefore separate the
syntactic from the semantic step.

\begin{definition}[The ask-extension]\label{def:ask}
Let $D$ be a domain presented by a language $L$ with a computable
numbering of its systems. Let $L^{+}$ be the language obtained by
adjoining to $L$ a single nullary gate $\ask$, as a symbol and
nothing more; $L^{+}$ is a decidable set of finite syntactic objects,
and we fix once and for all a computable numbering
$\nu:\NN\to L^{+}$ of it, extending that of $L$ and independent of any
semantics. Now let $\alpha$ be a partial function on $\NN$. The
\emph{ask-extension} $D^{+}_\alpha$ is the domain whose systems are
those of $L^{+}$, evaluated as in $L$ except that in the system
$\nu(p)$ every occurrence of $\ask$ evaluates to the mark if
$\alpha(p)\!\downarrow=1$, to the void if $\alpha(p)\!\downarrow\neq1$,
and diverges if $\alpha(p)\!\uparrow$. The numbering is fixed before
the semantics, so the definition is not circular: $p$ ranges over
$\nu$-indices of $L^{+}$, which are determined by syntax alone. A
domain $D$ is \emph{ask-closed} for $\alpha$ if $D^{+}_\alpha\subseteq
D$ up to a computable translation $t$ respecting both manifestation
and verdicts, i.e.\ $\nu(p)\in\M_{D^{+}_\alpha}\iff t(p)\in\M_D$ and
$\alpha(t(p))\simeq\alpha(p)$ for all $p$. (Literal inclusion, with
$t$ the identity, is the case one has in mind; the weaker formulation
is stated because it is what the proof of
Theorem~\ref{thm:cross}\eqref{cross:closed} uses.)
\end{definition}

\begin{lemma}[The gate distinction]\label{lem:gate}
Let $L$ contain, for each $b\in\{0,1\}$, a system of constant
behaviour $b$, and let $\alpha$ be any partial function on $\NN$.
Then there is a $\nu$-index $\delta$, depending only on $L$ and not on
$\alpha$, such that
\[
\nu(\delta)\in\M_{D^{+}_\alpha}
\quad\Longleftrightarrow\quad
\alpha(\delta)\!\downarrow=1 .
\]
\end{lemma}

\begin{proof}
The constant systems supply, for each $b$, a system of $L^{+}$ whose
behaviour is constantly $b$; this is what allows a verdict to be
realised as behaviour inside the domain rather than merely observed.
Fix, once and for all, a $\nu$-index $\delta$ of the system that
evaluates $\ask$ and then behaves as the constant it names. By
Definition~\ref{def:ask} the gate occurring in $\nu(\delta)$ is
evaluated at $\alpha(\delta)$, since $\delta$ is the index of the
system in which it occurs. Hence $\nu(\delta)$ manifests exactly when
its gate returns the mark, that is, exactly when
$\alpha(\delta)\!\downarrow=1$; and if $\alpha(\delta)\!\uparrow$ the
evaluation diverges and $\nu(\delta)$ does not manifest. As $\delta$
is fixed in advance, independently of $\alpha$, no uniformity
argument is required.
\end{proof}

\begin{remark}[The gate internalises the recursion theorem]\label{rem:internal}
No appeal to the recursion theorem occurs in the proof above, and this
is the point rather than an economy. In Theorem~\ref{thm:tri} the
self-application had to be manufactured: given $a$, one constructs
$\delta_a$ by s-m-n and a fixed point, and the construction is what
carries the diagonal. Definition~\ref{def:ask} instead builds the
self-application into the evaluation rule, so that the fixed point is
already present in the semantics and a single fixed syntactic object
suffices for every $\alpha$ at once. The gate is thus not a device
that permits diagonalisation but a language in which the recursion
theorem has been made a primitive; Theorem~\ref{thm:cross} is the
observation that adding this primitive is exactly what a decidable
domain cannot afford.
\end{remark}

\begin{theorem}[Crossing]\label{thm:cross}
Let $D$ be as in Lemma~\ref{lem:gate}.
\begin{enumerate}
\item\label{cross:partial} For every partial $X$-computable $\alpha$,
the domain $D^{+}_\alpha$ carries the full trichotomy of
Theorem~\ref{thm:tri} at the distinction $\delta$: $\alpha$ is not
simultaneously total, exhaustive and sound on $D^{+}_\alpha$,
and the three failure modes are exactly \emph{(i)}--\emph{(iii)}.
\item\label{cross:total} In particular, if $\alpha=v$ is total, case
\emph{(iii)} is excluded and $v$ fails exhaustiveness-with-soundness.
\item\label{cross:closed} No domain that is ask-closed for a correct
total adjudicator of itself is decidable.
\end{enumerate}
\end{theorem}

\begin{proof}
\eqref{cross:partial} and~\eqref{cross:total} are
Lemma~\ref{lem:gate} together with the case analysis of
Theorem~\ref{thm:tri}, the three cases of $\alpha(\delta)$ being
mutually exclusive and jointly exhaustive.
For~\eqref{cross:closed}, suppose $D$ were decidable and ask-closed
for its own correct total adjudicator $v$, so that $v(p)=1$ iff
$p\notin\M_D$. Ask-closure supplies $t$ with
$\nu(\delta)\in\M_{D^{+}_v}\iff t(\delta)\in\M_D$ and
$v(t(\delta))=v(\delta)$. Correctness of $v$ on $D$ gives
$v(t(\delta))=1\iff t(\delta)\notin\M_D$, and
Lemma~\ref{lem:gate} gives
$\nu(\delta)\in\M_{D^{+}_v}\iff v(\delta)=1$. Combining,
$v(\delta)=1\iff v(\delta)\neq1$.
\end{proof}

\begin{remark}[On the two hypotheses]\label{rem:crosshyp}
Both hypotheses of Lemma~\ref{lem:gate} do work. The constant systems
are what convert a verdict into behaviour: without them the gate could
be queried but its answer could not be made to determine whether the
system manifests, and no diagonal distinction would be expressible.
The computable numbering of $L$ is what makes the evaluation rule of
Definition~\ref{def:ask} well defined at all: without an index for the
system in which the gate occurs there is nothing for the gate to be
evaluated at. A domain presented without a numbering --- by an
arbitrary index set, say --- admits no ask-extension, and the theorem
genuinely concerns presented languages rather than sets of
distinctions.
\end{remark}

\begin{remark}[Self-address is not everywhere fatal]\label{rem:reflective}
Theorem~\ref{thm:cross} should not be read as showing that domains
containing a gate to their own adjudicator are impossible, but that
they are impossible under the constraints imposed here: the verdict is
a definite element of $\{\text{annulled},\text{exempt}\}$, and it is
required to be correct. Relaxing either escapes the conclusion. The
reflective oracles of Fallenstein, Taylor and Christiano
\cite{FTC15} are exactly the relaxation of the first: an oracle
answering questions about machines with access to that same oracle
exists, provided it is permitted to answer randomly precisely at those
queries on which the diagonal construction of Lemma~\ref{lem:gate}
would otherwise bite. Theorem~\ref{thm:jump} is the relaxation of the
second, in the other direction: a definite and correct adjudicator
exists, at the cost of one jump. The boundary is therefore not a wall
but a three-way trade: determinacy, correctness, and residence at the
level being adjudicated --- any two, never all three.
\end{remark}

\section{The bounded hierarchy}\label{sec:bounded}

Theorem~\ref{thm:cross} treats the gate as unrestricted: a verdict may
depend on a verdict, without bound. We now bound the depth, and find
that the crossing is not a single step but a limit. Throughout this
section we allow the gate a parameter --- $\ask(q)$, evaluated
at the verdict on the distinction of $\nu$-index $q$ --- of which the
gate of Definition~\ref{def:ask} is the case $q=p$. All statements
specialise to that case, and we record what the specialisation costs.

\begin{definition}[The iterates]\label{def:iterates}
Let $D$ be decidable and let $L^{+}$ adjoin the parametrised gate.
Define $a_0(p)=0$ for all $p$, and let $a_{k+1}(p)=1$ if the system
$\nu(p)$, evaluated with each occurrence of $\ask(q)$
returning $a_k(q)$, fails to manifest, and $a_{k+1}(p)=0$ otherwise.
Each $a_k$ is total computable, uniformly in $k$: the value
$a_{k+1}(p)$ depends on $a_k$ at the finitely many indices named in
$\nu(p)$, and $D$ is decidable. Write
$Q(p)=\{q:\ask(q)\text{ occurs in }\nu(p)\}$ and let $Q^{*}(p)$
be the closure of $\{p\}$ under $Q$.
\end{definition}

\begin{lemma}[Network form]\label{lem:network}
For each $p$ there is a Boolean function $F_p:\{0,1\}^{Q(p)}\to\{0,1\}$,
computable uniformly from $\nu(p)$, with
$a_{k+1}(p)=F_p\bigl(a_k\!\restriction\! Q(p)\bigr)$ for all $k$.
Hence the hierarchy is the synchronous update of a Boolean network on
the query graph, started at the all-zero state.
\end{lemma}

\begin{proof}
Fixing the values returned by the gates makes $\nu(p)$ a gate-free
system of $L$, whose manifestation $D$ decides; $F_p$ is the resulting
map from gate-values to the complement of that decision. Occurrences
of $\ask(q)$ with the same $q$ receive the same value, so
$F_p$ depends on $Q(p)$ and not on the number of occurrences.
\end{proof}

The dynamics of synchronous Boolean networks now govern the hierarchy,
and they are well understood; the content of the next theorem is the
translation, and the fact that the extreme behaviour is realised
inside the re-entry language of Section~\ref{sec:decidable} rather
than in some richer formalism.

\begin{theorem}[Finite closure]\label{thm:finiteclosure}
Suppose $|Q^{*}(p)|=n<\infty$. Then the sequence $(a_k(p))_{k\in\NN}$
is eventually periodic, with transient plus period at most $2^n$; the
eventual period is computable from $\nu(p)$; and the bound is
essentially attained: for infinitely many $n$ there is a system of the
re-entry language with $|Q^{*}(p)|=n$ whose sequence has period
$2^n-1$.
\end{theorem}

\begin{proof}
By Lemma~\ref{lem:network} the restriction of the hierarchy to
$Q^{*}(p)$ is a deterministic map on $\{0,1\}^{Q^{*}(p)}$, a set of
size $2^n$; a deterministic orbit in a finite set is eventually
periodic with transient plus period at most the size of the set, and
the orbit may be computed. For attainment, take the network of a
linear feedback shift register of length $n$ with a primitive feedback
polynomial and complemented feedback, i.e.\ the register whose new
leading bit is the negation of the exclusive-or of the tapped
positions. Complementation places the all-zero state on the long
cycle rather than at the fixed point of the linear register, so the
orbit from $a_0\equiv0$ has period $2^n-1$. Negation and disjunction
generate the exclusive-or, and both are primitive in the re-entry
language (Remark~\ref{rem:formalism}), so the network is presented by
a system of $L$.
\end{proof}

\begin{corollary}[Pure self-address]\label{cor:pureself}
If $Q(p)=\{p\}$ --- the gate of Definition~\ref{def:ask} --- then
$F_p$ is one of the four unary Boolean functions and the period is
$1$ or $2$. It is $2$ exactly when $F_p$ is negation, that is, exactly
at the diagonal distinctions of Lemma~\ref{lem:gate}. In particular
the pure self-address gate cannot produce long oscillation: the
richness of the hierarchy is a phenomenon of mutual, not reflexive,
query.
\end{corollary}

\begin{definition}[Presented networks]\label{def:presented}
A \emph{presented network} is a finite list
$N=\langle\nu(p_1),\dots,\nu(p_n)\rangle$ of systems of $L^{+}$ that is
\emph{query-closed}: every parameter $q$ occurring in an
$\ask(q)$ inside some $\nu(p_i)$ lies in
$\{p_1,\dots,p_n\}$. Its size $|N|$ is the total length of the list.
By Lemma~\ref{lem:network} the hierarchy restricted to $N$ is the
orbit of the all-zero state under a map on $\{0,1\}^n$, and
$Q^{*}(p_1)\subseteq\{p_1,\dots,p_n\}$.
\end{definition}

\begin{theorem}[Complexity of stabilisation]\label{thm:pspace}
The following problem is \PS-complete: given a presented
network $N$ of systems of the re-entry language, decide whether
$(a_k(p_1))_{k\in\NN}$ is eventually constant.
\end{theorem}

\begin{proof}
\emph{Membership.} A configuration is a vector in $\{0,1\}^n$, so
$n\leq|N|$ bits suffice to store it, and by Lemma~\ref{lem:network}
its successor is computable in space polynomial in $|N|$, since
computing $F_{p_i}$ amounts to evaluating a gate-free system of $L$,
which $D$ decides. The orbit enters its cycle within $2^n$ steps, and
the cycle has length at most $2^n$. The following procedure therefore
decides the problem: run $2^n$ steps; record the current value
$b=a(p_1)$; run a further $2^n$ steps, reporting \emph{no} if the
first coordinate ever differs from $b$, and \emph{yes} otherwise. Two
counters of $n+1$ bits, one configuration, and one bit are stored, so
the space used is polynomial in $|N|$.

\emph{Hardness.} We reduce from the acceptance problem for
deterministic Turing machines running in space $s(m)$ polynomial in
the input length, which is \PS-complete. Let $M$ be such a
machine and $w$ an input, $|w|=m$, and put $s=s(m)$. Encode an
instantaneous description of $M$ by $O(s)$ bits: the tape contents in
$s$ cells over a fixed alphabet, the head position in unary over $s$
bits, and the state in a fixed number of bits. Adjoin two further
bits, $\mathit{acc}$ and $\mathit{osc}$.

Define one node of the network per bit of this encoding. Each tape and
head bit is updated by the transition function of $M$, which is local:
the new content of cell $i$ depends on cells $i-1,i,i+1$, on the head
bits at those positions, and on the state bits. Each such dependence
is a Boolean function of boundedly many arguments, so is expressible by
a formula of size $O(1)$; the state bits depend on $O(s)$ arguments and
are expressible by formulas of size $O(s)$. Set $\mathit{acc}$ to
become and remain $1$ exactly when the state bits encode the accepting
state, and set
\[
\mathit{osc}\;\longmapsto\;\mathit{osc}\oplus\neg\,\mathit{acc},
\]
so that $\mathit{osc}$ alternates at every step at which $M$ has not
yet accepted, and is constant thereafter. Initialise the encoding to
the starting description of $M$ on $w$; this is achieved without
changing the all-zero start of the hierarchy: adjoin a counter of
$O(\log s)$ nodes which, during its first $O(s)$ steps, writes the
starting description of $M$ on $w$ into the encoding nodes and holds
the remaining update rules inert, releasing them once it saturates.
This adds $O(\log s)$ nodes and delays the simulation by $O(s)$ steps,
neither of which affects the analysis. We also make every terminal
configuration of $M$ --- accepting, rejecting, or stuck --- a fixed
point of the encoding update, so that the only source of continuing
change is $\mathit{osc}$.

If $M$ accepts $w$, then after the accepting step every bit of the
encoding is constant and $\mathit{acc}=1$, so $\mathit{osc}$ is
constant too and the orbit reaches a fixed point. If $M$ does not
accept, then $\mathit{acc}=0$ forever, so $\mathit{osc}$ alternates
forever and no coordinate-wise stabilisation occurs; making
$\mathit{osc}$ the first coordinate $p_1$, the sequence
$(a_k(p_1))_k$ is eventually constant if and only if $M$ accepts $w$.
This holds whether $M$ rejects, loops, or halts in a non-accepting
state, so no assumption on the behaviour of $M$ beyond its space bound
is needed.

The network has $O(s)$ nodes, each with a formula of size $O(s)$, so
$|N|$ is polynomial in $m$ and the reduction is computable in
logarithmic space. Finally, negation and disjunction are functionally
complete and are the two primitives of the re-entry language
(Remark~\ref{rem:formalism}), so every node function is realised by a
system of $L$, and the queries $\ask(q)$ supply exactly the
arguments. The list of these systems is query-closed by construction.
\end{proof}

\begin{remark}\label{rem:pspacescope}
The restriction to presented networks in Theorem~\ref{thm:pspace} is
not a technicality. For an arbitrary $p$ the closure $Q^{*}(p)$ may be
exponentially larger than $|\nu(p)|$, since a system of length $m$ may
name parameters of magnitude $2^m$; the configuration then does not
fit in space polynomial in the input, and the problem lies in
\textup{\textsc{ExpSpace}} and is not covered by Theorem~\ref{thm:pspace}. Presented networks are
the case in which the query graph is given explicitly, and are also
the case realised by the finite-state re-entry systems of
Proposition~\ref{prop:canonical}.
\end{remark}

\begin{proposition}[Unbounded closure]\label{prop:sigma2}
If $Q^{*}(p)$ is permitted to be infinite, the stabilisation problem
is $\Sigma^0_2$: the sequence $(a_k(p))_k$ is eventually constant iff
$\exists K\,\forall k\geq K\,(a_k(p)=a_K(p))$, and $a_k(p)$ is
computable uniformly in $k$ and $p$.
\end{proposition}

\begin{proof}
Immediate from Definition~\ref{def:iterates} and the form of the
condition.
\end{proof}

\begin{remark}[What the hierarchy shows]\label{rem:hierarchy}
The crossing of Theorem~\ref{thm:cross} is therefore not a
discontinuity. Every finite level is decidable, and cheaply so in the
depth: the cost of level $k$ is $k$ evaluations, not $k$ nested
searches, because all occurrences of a gate with the same parameter
receive the same value. What fails at the limit is not decidability
but convergence. The diagonal distinctions are exactly the points of
oscillation, and by Corollary~\ref{cor:pureself} they oscillate with
period two.
\end{remark}

\begin{remark}[The mean of the oscillation]\label{rem:mean}
This suggests a reading of the reflective oracles of \cite{FTC15}
which we state as an observation rather than a theorem. Where the
hierarchy oscillates with period $\pi$, the natural single value to
assign a distinction is the mean frequency of the mark along the
cycle, a rational with denominator dividing $\pi$. For the diagonal
distinctions $\pi=2$ and the mean is $1/2$ --- which is exactly the
value a reflective oracle is forced to return at the liar machine, and
the reason its answer there must be a fair coin. On this reading the
randomisation is not a device for evading diagonalisation but the
assignment of a limiting frequency to a sequence that fails to
converge. Whether the correspondence is exact --- whether every
reflective oracle agrees with the mean frequency of the hierarchy
wherever the latter is defined --- we leave open; it is
Question~\ref{q:bounded} in its sharpened form.
\end{remark}

The boundary now has two syntactic descriptions, and they differ.
Stated crudely, decidable annulment survives multiplicity, iteration
and finite re-entry, and dies at the first gate of self-address.
Stated with the hierarchy in view, it does not die there at all: it
survives every finite depth of self-address and dies only in the
limit, and what dies is convergence rather than decidability. The next
section gives a third description, in degrees, which is insensitive to
this refinement --- adjudication of a level lives strictly above the
level, whatever the route by which one arrives at that level.

\section{The cost of adjudication: exactly one jump}\label{sec:jump}

\begin{theorem}\label{thm:jump}
Let $X\subseteq\NN$. A total function $v$ decides $\M^X$
(i.e.\ $v(p)=1\iff p\notin\M^X$) if and only if $v\geq_T X'$; and the
Turing degrees of such $v$ are exactly the degrees $\geq\deg(X')$. In
particular the least Turing degree of a total correct adjudicator
of the $X$-standard domain is $\deg(X')$: adjudication costs exactly
one jump, at every level.
\end{theorem}

\begin{proof}
$\M^X$ is $\Sigma^{0,X}_1$-complete (Proposition~\ref{prop:ledger}),
hence $\M^X\equiv_T X'$. Any $v$ deciding $\M^X$ computes its
characteristic function, so $v\geq_T\M^X\equiv_T X'$; conversely
$\chi_{\M^X}\leq_T X'$ is such a $v$.

For the realisation of every degree $\geq\deg(X')$, note that a $v$
deciding $\M^X$ is pinned down on $\Led^X$ and free only in its
non-$1$ values on $\M^X$; the coding must therefore be carried out
there. Let $q$ be a strictly increasing computable enumeration of
pairwise distinct indices, obtained by padding, of the program that
outputs $1$ on input $0$ without consulting the oracle; then
$R=\operatorname{range}(q)$ is an infinite computable subset of $\M^X$
for every $X$. Given $A\geq_T X'$, put $v(p)=1$ for $p\notin\M^X$,
$v(q(n))=2$ if $n\in A$ and $0$ if $n\notin A$, and $v(p)=0$ for
$p\in\M^X\setminus R$. Then $v$ decides $\M^X$; $A\leq_T v$ since
$n\in A\iff v(q(n))=2$; and $v\leq_T A\oplus X'\equiv_T A$, since $X'$
settles membership in $\M^X$ and $R$ is computable. Hence
$v\equiv_T A$.
\end{proof}

\begin{corollary}[Escalation]\label{cor:escalation}
There is no oracle $X$ such that some $X$-computable total $v$ decides
$\M^X$. Consequently, if a domain is closed under $X$-computation and
its correct total adjudicator is required to lie in the same degree
$\deg(X)$, no such adjudicator exists; adjudication of a level always
resides strictly above the level.
\end{corollary}

\begin{proof}
$X$-computable total $v$ implies $v\leq_T X<_T X'\leq_T v$ by
Theorem~\ref{thm:jump}, a contradiction.
\end{proof}

\section{The effective structure of manifestation}\label{sec:structure}

\begin{proposition}\label{prop:complete}
$\M$ is $\Sigma^0_1$-complete, indeed $1$-complete; $\M$ is creative;
and $\M$ is computably isomorphic to $K$.
\end{proposition}

\begin{proof}
$\M\in\Sigma^0_1$ by the normal form theorem. By s-m-n there is an
injective total computable $h$ with
$\varphi_{h(e,x)}(0)=1\iff\varphi_e(x)\!\downarrow$ (simulate
$\varphi_e(x)$, output $1$ on convergence; injectivity by padding), so
$K_0\leq_1\M$ and $\M$ is $1$-complete. A $1$-complete c.e.\ set is
creative \cite{Pos44,Rog67}, and creative sets are computably
isomorphic to $K$ by Myhill's theorem \cite{Myh55}.
\end{proof}

\begin{proposition}[The ledger]\label{prop:ledger}
$\Led=\NN\setminus\M$ is $\Pi^0_1$-complete and productive: there is a
total computable $\psi$ such that whenever $W_e\subseteq\Led$, we have
$\psi(e)\in\Led\setminus W_e$. Relativised: $\M^X$ is
$\Sigma^{0,X}_1$-complete and $\Led^X$ is $\Pi^{0,X}_1$-complete and
$X$-productive, for every $X$.
\end{proposition}

\begin{proof}
Immediate from Proposition~\ref{prop:complete} and the
Post--Myhill theory of creative sets \cite{Pos44,Rog67}; the
relativised claims follow by relativising the proofs, using
Corollary~\ref{cor:rel} for uniformity.
\end{proof}

\begin{remark}
Productivity is the effective content of the slogan that any attempted
census of the annulled generates its own counterexample: from an index
of any c.e.\ list of true annulments, $\psi$ computes a true annulment
missing from the list. Together with Proposition~\ref{prop:complete},
the manifestation problem is, up to a computable permutation of
indices, \emph{the} halting problem.
\end{remark}

\begin{proposition}[No effective complementation]\label{prop:nocomp}
There is no total computable $h$ with $h(p)\in\M\iff p\notin\M$ for
all $p$.
\end{proposition}

\begin{proof}
Such an $h$ would witness $\Led\leq_m\M$, placing the
$\Pi^0_1$-complete set $\Led$ in $\Sigma^0_1$, a contradiction.
\end{proof}

Alongside $\M$, consider
\[
E_0=\{p:\varphi_p(0)\!\downarrow=0\},
\]
the \emph{explicitly exempt} distinctions: those whose computation
halts without the mark. $\M$ and $E_0$ are disjoint c.e.\ sets, and
$E_0\subseteq\Led$.

\begin{theorem}[Effective inseparability]\label{thm:ei}
$(\M,E_0)$ is an effectively inseparable pair: there is a total
computable $f$ such that whenever $\M\subseteq W_i$,
$E_0\subseteq W_j$ and $W_i\cap W_j=\varnothing$, then
$f(i,j)\notin W_i\cup W_j$.
\end{theorem}

\begin{proof}
By s-m-n and the recursion theorem with parameters there is a total
computable $f$ such that $\varphi_{f(i,j)}(0)$ behaves as follows: it
dovetails the enumerations of $W_i$ and $W_j$, searching for the
number $f(i,j)$ itself; if $f(i,j)$ is found first in $W_j$ it outputs
$1$; if first in $W_i$ it outputs $0$; if never, it diverges. Write
$n=f(i,j)$ and suppose $\M\subseteq W_i$, $E_0\subseteq W_j$,
$W_i\cap W_j=\varnothing$, and $n\in W_i\cup W_j$; then the search
terminates. If it terminates through $W_j$, then $\varphi_n(0)=1$, so
$n\in\M\subseteq W_i$, while also $n\in W_j$: this contradicts
disjointness. If it terminates through $W_i$, then $\varphi_n(0)=0$,
so $n\in E_0\subseteq W_j$, and again $n\in W_i\cap W_j$. Hence
$n\notin W_i\cup W_j$.
\end{proof}

\begin{corollary}[No computable separation]\label{cor:insep}
There is no computable $C$ with $\M\subseteq C$ and
$C\cap E_0=\varnothing$. In annulment terms: no total computable
adjudicator sorts the manifesting from the explicitly exempt, however
it is allowed to classify the silent distinctions. This strengthens
Corollary~\ref{cor:nodecider}: not only is $\M$ undecidable, it
cannot even be computably separated from its certified-exempt part.
\end{corollary}

\begin{proof}
If $C$ were such a set, then $C$ and $\NN\setminus C$ are c.e.,
$\M\subseteq C$, $E_0\subseteq\NN\setminus C$, and they are
disjoint; Theorem~\ref{thm:ei} yields
$n\notin C\cup(\NN\setminus C)=\NN$, which is absurd.
\end{proof}

\begin{remark}
By Smullyan's isomorphism theory for effectively inseparable pairs of
disjoint c.e.\ sets \cite{Smu61}, the pair $(\M,E_0)$ is computably
isomorphic to the canonical pair
$(\{e:\varphi_e(e)\!\downarrow=1\},\{e:\varphi_e(e)\!\downarrow=0\})$:
manifestation versus explicit exemption is, up to a computable
permutation of indices, \emph{the} canonical effectively inseparable
dichotomy.
\end{remark}

\section{Index sets: the exact complexity of the adjudicator properties}\label{sec:indexsets}

Write $\mathrm{S}=\{a:\text{$a$ is sound}\}$,
$\mathrm{T}=\{a:\text{$a$ is total}\}$,
$\mathrm{ET}=\{a:\forall p\,\varphi_a(p)\!\downarrow=1\}$, and let
$H_i=\{a:\text{case (i) of Theorem~\ref{thm:tri} obtains at }
\delta_a\}$ and similarly $H_{ii}$, $H_{iii}$.

\begin{theorem}\label{thm:index}
\begin{enumerate}
\item $\mathrm{S}$ is $\Pi^0_1$-complete.
\item $\mathrm{T}$, $\mathrm{ET}$, and $\mathrm{T}\cap\mathrm{S}$ are
$\Pi^0_2$-complete.
\item $H_i$ and $H_{ii}$ are $\Sigma^0_1$-complete; $H_{iii}$ is
$\Pi^0_1$-complete.
\end{enumerate}
\end{theorem}

\begin{proof}
(1) $a\in\mathrm{S}$ iff
$\forall p\,\neg(\varphi_a(p)\!\downarrow=1\wedge\varphi_p(0)\!\downarrow=1)$,
a $\Pi^0_1$ condition. For hardness we reduce the complement of $K$:
by s-m-n, from $e$ compute an index $a_e$ which, on any input,
simulates $\varphi_e(e)$ and outputs $1$ upon convergence, diverging
otherwise. If $e\notin K$ then $a_e$ is everywhere silent, hence
vacuously sound; if $e\in K$ then $a_e$ annuls every distinction,
including manifesting ones (which exist), so $a_e\notin\mathrm{S}$.
Thus $e\notin K\iff a_e\in\mathrm{S}$: a many-one reduction of the
complement of $K$ to $\mathrm{S}$.

(2) $\mathrm{T}=\mathrm{Tot}$ is $\Pi^0_2$-complete classically
\cite{Rog67}. $\mathrm{ET}$ is $\Pi^0_2$ ($\forall p\,\exists
s\,[\varphi_{a,s}(p)\!\downarrow=1]$); for hardness, given $e$ let
$a_e(p)$ simulate $\varphi_e(p)$ and output $1$ on convergence: then
$a_e\in\mathrm{ET}\iff e\in\mathrm{Tot}$. For
$\mathrm{T}\cap\mathrm{S}$: membership is
$\Pi^0_2\wedge\Pi^0_1=\Pi^0_2$; for hardness, given $e$ let $a_e(p)$
simulate $\varphi_e(p)$ and output $0$ (exempt) on convergence: $a_e$
never annuls, hence is vacuously sound, and is total iff
$e\in\mathrm{Tot}$.

(3) $H_i\cup H_{ii}=\{a:\varphi_a(\delta_a)\!\downarrow\}$ and each of
$H_i$, $H_{ii}$ is manifestly $\Sigma^0_1$; $H_{iii}$ is the
complement of a $\Sigma^0_1$ set, hence $\Pi^0_1$. Hardness of $H_i$:
given $\langle e,x\rangle$, let $a$ simulate $\varphi_e(x)$ on any
input and output $1$ on convergence; then $a\in H_i\iff\langle
e,x\rangle\in K_0$. Hardness of $H_{ii}$: the same with output $0$.
Hardness of $H_{iii}$: the same construction reduces the complement of
$K_0$ to $H_{iii}$.
\end{proof}

Case (i) is where a total mission-complete adjudicator lands: (T)
excludes case (iii), and (M) excludes case (ii), since
by~\eqref{tri:fix} an exempted $\delta_a$ lies outside $\M$ and
mission-completeness would then require its annulment. The membership
of $H_i$ is then witnessed by a finite computation --- the failure of
soundness is not merely true but effectively observable.

\section{Relation to categorical and numbering-theoretic frameworks}\label{sec:lawvere}

Diagonal arguments admit a well-known unifying treatment, initiated by
Lawvere \cite{Law69} and developed by Yanofsky \cite{Yan03} and Bauer
\cite{Bau17}: in a cartesian closed category, a weakly point-surjective
$A\to Y^A$ forces every endomorphism of $Y$ to have a fixed point, and
the contrapositive subsumes the arguments of Cantor, Russell, Tarski
and Turing. It is therefore reasonable to ask what, if anything, the
present framework adds. We record the answer, which is a matter of
which side of a known divide the constructions fall on, rather than of
novelty.

The divide is between Kleene's two recursion theorems. The first is a
statement about computation at higher types: it concerns effective
operations, that is, transformations induced by \emph{extensional}
functions --- those $f$ with $\varphi_a=\varphi_b\Rightarrow
\varphi_{f(a)}=\varphi_{f(b)}$ --- and it delivers least fixed points
by iteration. The second is a first-order, diagonal statement about
indices, and it places no extensionality requirement on anything.
Lawvere's theorem corresponds to a restricted form of the first
\cite{Kav17}. That the two are genuinely different, even where both
apply, is Rogers' observation that there is an extensional function
whose fixed point as delivered by the second recursion theorem is not
the least one \cite[\S11-XIII]{Rog67}.

Every construction of this paper falls on the second, intensional
side. In Theorem~\ref{thm:tri} the adjudicator $a$ is applied to the
\emph{index} $\delta_a$, not to the function $\varphi_{\delta_a}$, and
the theorem is asserted for arbitrary $a$ --- in particular for
adjudicators that inspect program text, such as one annulling exactly
the distinctions of even index. Such an $a$ induces no effective
operation and lies outside the scope of the Lawvere scheme, which
requires the transformation to depend on extension alone.
The generalisation of Remark~\ref{rem:numberings} has the same
character, Ershov's theorem being a generalisation of the second
recursion theorem rather than of the first; and the gate of Lemma~\ref{lem:gate} queries an adjudicator
at an index, not at a behaviour.

We should be equally clear about what does \emph{not} distinguish the
framework. Partiality alone does not: with lifted objects and the
multi-valued formulation of \cite{Bau17}, the Lawvere scheme
accommodates divergence, and the interest of the third horn lies in
its pricing of silence (Remark~\ref{rem:numberings}) rather than in
any categorical inaccessibility. Nor is the intensional side itself
unaxiomatised: the exposures of \cite{Kav17} give a categorical
account of intensional fixed points from which abstract forms of the
theorems of G\"odel, Tarski and Rice are derived, and
Theorem~\ref{thm:tri} would be at home there.

What is not on offer in any of these frameworks is a boundary. They
unify statements of impossibility; they do not say which domains admit
a total correct adjudicator and which do not, nor identify the single
syntactic feature that separates the two. That is the content of
Section~\ref{sec:decidable}, and it is the claim of this paper.

\section{Questions}\label{sec:questions}

The two questions below correspond to two sides of the boundary:
the decidable interior and the crossing itself.

\begin{question}\label{q:cert}
Characterise the domains $D$, presented by sublanguages of the
re-entry language of Section~\ref{sec:decidable}, for which the
canonical adjudicator is computable with certificates of polynomial
size; equivalently, locate the certificate-complexity boundary inside
the decidable side.
\end{question}

\begin{question}\label{q:bounded}
Section~\ref{sec:bounded} shows that each finite level of the
hierarchy is decidable and that stabilisation at finite closure is
\PS-complete. Two things are left open. First, is the
$\Sigma^0_2$ bound of Proposition~\ref{prop:sigma2} sharp --- is
stabilisation $\Sigma^0_2$-complete when the query closure is
unbounded? Second, is the correspondence of Remark~\ref{rem:mean}
exact: does every reflective oracle in the sense of \cite{FTC15}
agree, at every distinction where the hierarchy is eventually
periodic, with the mean frequency of the mark along its cycle? An
affirmative answer would exhibit the randomisation of reflective
oracles as the assignment of limiting frequencies rather than as an
independent construction.
\end{question}

\section*{Declaration on the use of AI tools}
All mathematical content of this paper --- the framework, the
definitions, the theorems, and their proofs --- is the author's own. A
large language model (Claude, Anthropic) was used as an editorial aid
in the preparation of the manuscript: for referee-style review of
drafts, for verification of bibliographic data, for consistency checks
on cross-references and terminology, for suggestions on the wording of
proofs, and for language polishing. All such output was
verified by the author, who takes full responsibility for the entire
content of the paper.


\begin{thebibliography}{99}\setlength\itemsep{1pt}

\bibitem{Bau17} A.~Bauer, On fixed-point theorems in synthetic
computability, \emph{Tbilisi Mathematical Journal} 10 (2017), no.~3,
167--181.

\bibitem{Buc62} J.~R.~B\"uchi, On a decision method in restricted
second order arithmetic, in: \emph{Logic, Methodology and Philosophy
of Science (Proc.\ 1960 Congr.)}, Stanford University Press, 1962,
1--11.

\bibitem{Elg61} C.~C.~Elgot, Decision problems of finite automata
design and related arithmetics, \emph{Trans.\ Amer.\ Math.\ Soc.} 98
(1961), 21--51.

\bibitem{Ers73} Yu.~L.~Ershov, Theorie der Numerierungen I,
\emph{Zeitschrift f\"ur mathematische Logik und Grundlagen der
Mathematik} 19 (1973), 289--388.

\bibitem{FTC15} B.~Fallenstein, J.~Taylor and P.~F.~Christiano,
Reflective oracles: a foundation for game theory in artificial
intelligence, in: W.~van der Hoek, W.~Holliday and W.~Wang (eds.),
\emph{Logic, Rationality, and Interaction (LORI 2015)}, Lecture Notes
in Computer Science 9394, Springer, 2015, 411--415.

\bibitem{Kav17} G.~A.~Kavvos, \emph{On the Semantics of Intensionality
and Intensional Recursion}, D.Phil.\ thesis, University of Oxford,
2017; arXiv:1712.09302 [cs.LO].

\bibitem{KV80} L.~H.~Kauffman and F.~J.~Varela, Form dynamics,
\emph{Journal of Social and Biological Structures} 3 (1980), 171--206.


\bibitem{Kle38} S.~C.~Kleene, On notation for ordinal numbers,
\emph{Journal of Symbolic Logic} 3 (1938), 150--155.

\bibitem{Law69} F.~W.~Lawvere, Diagonal arguments and cartesian closed
categories, in: \emph{Category Theory, Homology Theory and their
Applications II}, Lecture Notes in Mathematics 92, Springer, 1969,
134--145; reprinted in \emph{Reprints in Theory and Applications of
Categories} 15 (2006), 1--13.

\bibitem{Myh55} J.~Myhill, Creative sets, \emph{Zeitschrift f\"ur
mathematische Logik und Grundlagen der Mathematik} 1 (1955), 97--108.


\bibitem{Pos44} E.~L.~Post, Recursively enumerable sets of positive
integers and their decision problems, \emph{Bulletin of the American
Mathematical Society} 50 (1944), 284--316.

\bibitem{Pre29} M.~Presburger, \"Uber die Vollst\"andigkeit eines
gewissen Systems der Arithmetik ganzer Zahlen, in welchem die Addition
als einzige Operation hervortritt, in: \emph{Comptes Rendus du I
Congr\`es des Math\'ematiciens des Pays Slaves}, Warszawa, 1929,
92--101, 395.

\bibitem{Rog67} H.~Rogers, Jr., \emph{Theory of Recursive Functions
and Effective Computability}, McGraw-Hill, New York, 1967.

\bibitem{Sif26} P.~Sifnaios, Weak essentially undecidable theories of
hereditarily finite multisets, arXiv:2607.11367 [math.LO], 2026.

\bibitem{Smu61} R.~M.~Smullyan, \emph{Theory of Formal Systems},
Annals of Mathematics Studies 47, Princeton University Press,
Princeton, 1961.

\bibitem{Soa87} R.~I.~Soare, \emph{Recursively Enumerable Sets and
Degrees}, Springer, Berlin, 1987.

\bibitem{SB69} G.~Spencer-Brown, \emph{Laws of Form}, George Allen and
Unwin, London, 1969.

\bibitem{Var75} F.~J.~Varela, A calculus for self-reference,
\emph{International Journal of General Systems} 2 (1975), 5--24.

\bibitem{WB95} P.~Wolper and B.~Boigelot, An automata-theoretic
approach to Presburger arithmetic constraints, in: \emph{Static
Analysis (SAS~'95)}, Lecture Notes in Computer Science 983, Springer,
1995, 21--32.

\bibitem{Yan03} N.~S.~Yanofsky, A universal approach to
self-referential paradoxes, incompleteness and fixed points,
\emph{Bulletin of Symbolic Logic} 9 (2003), no.~3, 362--386.

\end{thebibliography}
\end{document}